\documentclass[oneside]{amsart}
\usepackage[unicode,hidelinks]{hyperref}

\usepackage{xypic}
\usepackage{stix}
\usepackage[utf8]{inputenc}

\newcommand{\forces}{\Vdash}
\newtheorem{theorem}[equation]{Theorem}

\newtheorem{assumption}[equation]{Assumption}

\newtheorem{lemma}[equation]{Lemma}
\newtheorem{facts}[equation]{Facts}

\theoremstyle{remark}
\newtheorem{remark}[equation]{Remark}

\theoremstyle{definition}
\newtheorem{definition}[equation]{Definition}
\numberwithin{equation}{section}

\newcommand{\textdef}[1]{\boldmath\textbf{#1}\unboldmath}

\DeclareMathOperator{\crit}{cr}
\DeclareMathOperator{\mySII}{\boxplus_2}

\DeclareMathOperator{\myXi}{\circledcirc_i}
\DeclareMathOperator{\myXI}{\circledcirc_1}
\DeclareMathOperator{\myXII}{\circledcirc_2}
\DeclareMathOperator{\myXIII}{\circledcirc_3}

\DeclareMathOperator{\cf}{cf}
\DeclareMathOperator{\supp}{supp}
\DeclareMathOperator{\Rel}{R}
\DeclareMathOperator{\Ri}{R_i}
\DeclareMathOperator{\RI}{R_1}
\DeclareMathOperator{\RII}{R_2}
\DeclareMathOperator{\Ra}{R_3}
\newcommand{\Pa}{P^4}
\newcommand{\Pai}{P^i}
\newcommand{\PaV}{P^5}
\newcommand{\PaVI}{P^6}
\newcommand{\PaVII}{P^7}

\DeclareMathOperator{\cov}{cov}
\DeclareMathOperator{\cof}{cof}
\DeclareMathOperator{\non}{non}
\DeclareMathOperator{\add}{add}

\newcommand{\covnull}{\cov(\mathcal N)}
\newcommand{\cofnull}{\cof(\mathcal N)}
\newcommand{\addnull}{\add(\mathcal N)}
\newcommand{\nonnull}{\non(\mathcal N)}
\newcommand{\covmeager}{\cov(\mathcal M)}
\newcommand{\cofmeager}{\cof(\mathcal M)}
\newcommand{\addmeager}{\add(\mathcal M)}
\newcommand{\nonmeager}{\non(\mathcal M)}
\newcommand{\BUP}[1]{#1^U}

\date{\today}
\title{Compact Cardinals and Eight Values in Cicho\'{n}'s Diagram}
\author{Jakob Kellner}
\email{jakob.kellner@tuwien.ac.at}
\author{Anda Ramona T{\u{a}}nasie}
\email{anda-ramona.tanasie@tuwien.ac.at}
\author{Fabio Elio Tonti}
\email{fabio.tonti@tuwien.ac.at}
\address{Institute of Discrete Mathematics and Geometry\\
Technische Universität Wien (TU Wien).}
\begin{document}

\begin{abstract}
Assuming  three strongly compact cardinals, it is consistent that
\[
\aleph_1 < \addnull < \covnull < \mathfrak{b} < \mathfrak{d} < \nonnull < \cofnull < 2^{\aleph_0}.\]
Under the same assumption, it is consistent that
\[
\aleph_1 < \addnull < \covnull < \nonmeager < \covmeager < \nonnull < \cofnull < 2^{\aleph_0}.\]
\end{abstract}

\maketitle
\section*{Introduction}
We assume the reader is familiar with the definitions and some basic
properties (which can all be found, e.g., in~\cite{BJ}) of the cardinal 
characteristics in Cicho\'n's diagram:
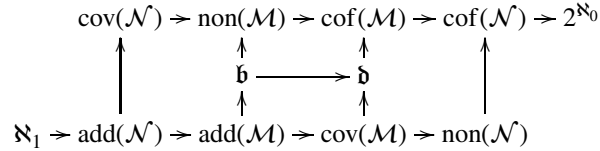
\[
\xymatrix@=1.6ex{
           & \covnull\ar[r]        & \nonmeager \ar[r]      &  \cofmeager \ar[r]     & \cofnull\ar[r]  &2^{\aleph_0} \\
           &                    & \mathfrak b\ar[r]\ar[u]  &  \mathfrak d\ar[u] &              &\\ 
\aleph_1\ar[r] & \addnull\ar[r]\ar[uu] & \addmeager\ar[r]\ar[u] &  \covmeager\ar[r]\ar[u]& \nonnull\ar[uu] &
}
\]
An arrow between $\mathfrak x$ and $\mathfrak y$ indicates that ZFC proves
$\mathfrak x\le \mathfrak y$. 
Moreover, $\max(\mathfrak d,\nonmeager)=\cofmeager$ and $\min(\mathfrak b,\covmeager)=\addmeager$.
These are the only ``simple'' restrictions in the following sense: every assignment of $\aleph_1$ and $\aleph_2$ to the entries of Cicho\'{n}'s diagram that  honors these restrictions can be shown to be consistent.
It is more challenging to get 
more than two simultaneously different values, for recent progress in this direction
see, e.g., \cite{MR3047455,MR3513558,five}.

This paper consists of two parts: In the first one, 
we present a finite support ccc iteration $\Pa$ forcing that
$\aleph_1 < \addnull < \covnull < \mathfrak{b}<\mathfrak{d}=2^{\aleph_0}$
(and actually something stronger, cf.~Lemmas~\ref{lem:goodreward3} and~\ref{lem:Pcovnullislarge}). This is nothing new: The forcing and all required 
properties were presented in~\cite{MR3047455}.
We recall all the facts that are required for our result,
in a form convenient for our purposes.

In the second part, 
we investigate the (iterated) Boolean 
ultrapower $\PaVII$  of $\Pa$. Assuming three strongly compact cardinals, this ultrapower (again a finite support ccc iteration) forces
\[
\aleph_1 < \addnull < \covnull < \mathfrak{b} < \mathfrak{d} < \nonnull < \cofnull < 2^{\aleph_0},\]
i.e., we get the following values in the diagram (for some increasing cardinals $\lambda_i$):
\newcommand{\mye}{*+[F.]{\phantom{\lambda}}}
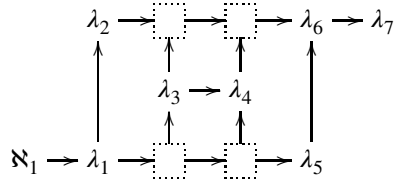
\[
\xymatrix@=2.5ex{
           & \lambda_2\ar[r]        & \mye \ar[r]      &  \mye \ar[r]     & \lambda_6\ar[r]  &\lambda_7 \\
           &                    & \lambda_3\ar[r]\ar[u]  &  \lambda_4\ar[u] &              &\\ 
\aleph_1\ar[r] & \lambda_1\ar[r]\ar[uu] & \mye\ar[r]\ar[u] &  \mye\ar[r]\ar[u]& \lambda_5\ar[uu] &
}
\]    
It seems unlikely that the large cardinals assumption is actually needed, but we would expect 
a proof without it to be considerably more complicated.

The kind of Boolean ultrapower that we use was 
investigated in~\cite{MR0300887}, 
and recently applied, e.g., in~\cite{MR3451934} and~\cite{F1611} (where a
Boolean ultrapower of a forcing notion is applied to cardinal characteristics of the reals).                                                                             Recently Shelah developed a method of using Boolean ultrapowers to control characteristics in Cicho\'n's diagram.
The current paper is a relatively simple application of these methods.
A more complicated one, in an upcoming paper~\cite{1122} by 
Goldstern, Shelah and the first author, shows that all entries in Cicho\'ns diagram can be pairwise different.

\subsection*{Acknowledgments}
We are grateful to Martin Goldstern and Saharon Shelah for valuable discussions.

We would like to thank an anonymous referee for their quick, insightful and kind review, and for pointing out a few embarrassing typos and mistakes.

Supported by the Austrian Science Fund (FWF): P26737 and P30666.
The second and third authors are both recipients of a DOC Fellowship of the Austrian Academy of Sciences at the Institute of Discrete Mathematics and Geometry, TU Wien.

\section{The initial forcing}\label{sec:partA}


\subsection{Good iterations}\label{sec:bla1}

The forcing $\Pa$ 
we are about to define has many pleasant properties 
because it is 
``good'', a notion
first explored 
in
~\cite{MR1071305} and
~\cite{MR1129144}.
We now recall the basic facts of good iterations, and specify the instances 
of the relations we use.

\begin{assumption}
We will consider binary relations $\Rel$ on $X=\omega^\omega$ (or on $X=2^\omega)$ that satisfy the following: There are relations $\Rel^n$ such that
$\Rel=\bigcup_{n\in\omega} \Rel^n$, each $\Rel^n$ is a closed subset (and in fact absolutely defined)
of $X\times X$, and for  $g\in X$
and $n\in\omega$, the set $\{f\in X:\, f\Rel^n g\}$ is nowhere dense. 
Also, for all $g\in X$ there is some $f\in X$ with $f \Rel g$.
\end{assumption} 

We will actually use another space as well,
the space $\mathcal C$ 
of strictly positive rational sequences $(q_n)_{n\in\omega}$ such that
$\sum_{n\in\omega}q_n\le 1$.
It is easy to see that $\mathcal C$ is homeomorphic to
$\omega^\omega$, when
we equip the rationals with the discrete topology 
and use the product topology.

We use the following instances of relations $\Rel$ 
on $X$; it is easy
to see that they all satisfy the assumption (in case of $X=\mathcal C$ we use the homeomorphism mentioned above):
\begin{definition}
\begin{itemize}
\item[1.] $X=\mathcal C$: $f \RI g$  if
$(\forall^*n\in\omega)\, f(n)\le g(n)$.
\\
(We use ``$\forall ^* n\in \omega$'' for ``$(\exists n_0\in\omega)\,(\forall n>n_0)$''.)
\item[2.] $X=2^\omega$: $f \RII g$ if $(\forall^* n\in\omega)\, f\restriction I_n\neq g\restriction I_n$, 
\\
where 
$(I_n)_{n\in\omega}$ is the increasing interval partition of $\omega$ with $|I_n|=2^{n+1}$.
\item[3.] $X= \omega^\omega$: $f \Ra g$  if $(\forall^* n\in\omega)\, f(n)\le g(n)$.
\end{itemize}
\end{definition}

We say ``$f$ is bounded by $g$'' if $f\Rel g$; and, for $\mathcal Y\subseteq \omega^\omega$, 
``$f$ is bounded by $\mathcal Y$'' if
$(\exists y\in \mathcal Y)\, f \Rel y$. We say ``unbounded'' for ``not bounded''. (I.e., $f$ is unbounded by $\mathcal Y$ if $(\forall y\in \mathcal Y)\,\lnot f\Rel y$.) 
We call $\mathcal X$  an $\Rel$-unbounded family, if $\lnot (\exists g)\,(\forall x\in \mathcal X)\,x\Rel g$, and an $\Rel$-dominating family if $(\forall f)\,(\exists x\in \mathcal X)\, f\Rel x$. Let $\mathfrak{b}_i$ be the minimal size of an $\Ri$-unbounded family, 
and $\mathfrak{d}_i$ of an $\Ri$-dominating family.

We only need the following connection between $\Ri$ and the cardinal characteristics: 
\begin{lemma}\label{lem:connection}
\begin{itemize}
\item[1.]
$\addnull=\mathfrak{b}_1$ and $\cofnull=\mathfrak{d}_1$.
\item[2.]
$\covnull\le\mathfrak{b}_2$ and $\nonnull\ge\mathfrak{d}_2$.
\item[3.]
$\mathfrak{b}=\mathfrak{b}_3$ and $\mathfrak{d}=\mathfrak{d}_3$.
\end{itemize}
\end{lemma}
\begin{proof}
(3) holds by definition. (1) can be found in~\cite[6.5.B]{BJ}.
To prove (2), note that 
for fixed $g\in 2^\omega$ the set $\{f\in 2^\omega: \lnot g\RII f\} $
is a null set, call it $N_g$. 
%
Let $\mathcal G$ be an $\RII$-unbounded family.
Then $\{N_g:\, g\in\mathcal{G}\}$ covers $2^\omega$:
Fix $f\in 2^\omega$.
As $f$ does not bound $\mathcal{G}$, there is some 
$g\in \mathcal{G}$ unbounded by $f$, i.e., $f\in N_g$.
Let $X$ be a non-null set. 
Then $X$ is $\RII$-dominating: For any $g\in 2^\omega$
there is some $x\in X\setminus N_g$, i.e., $g \RII x$.
%
%
\end{proof}

%

\begin{definition}\cite{MR1071305}\label{def:good}
Let $P$ be a ccc forcing, $\lambda$ an uncountable regular cardinal, and $\Rel$ as above.
$P$ is \textdef{\math(\Rel,\lambda)$-good},
if for each $P$-name $r\in\omega^\omega$ there is (in $V$) a nonempty
set $\mathcal Y\subseteq \omega^\omega$ of size ${<}\lambda$
such that every $f$ (in $V$) that is $\Rel$-unbounded by $\mathcal Y$ is forced to be $\Rel$-unbounded by $r$ as well.
\end{definition}

Note that $\lambda$-good trivially implies $\mu$-good if $\mu\ge\lambda$
are regular.

How do we get good forcings? Let us just quote the following results:
\begin{lemma}\label{lem:gettinggood}
A FS iteration of Cohen forcing is good for any $(\Rel,\lambda)$, and
the composition of two $(\Rel,\lambda)$-good forcings is $(\Rel,\lambda)$-good.
\\
Assume that $(P_\alpha,Q_\alpha)_{\alpha<\delta}$ is a FS ccc iteration.
Then $P_\delta$ is $(\Rel,\lambda)$-good, if each $Q_\alpha$ is forced to satisfy the following:
\begin{itemize}
\item[1.] For $\Rel=\RI$:  $|Q_\alpha|<\lambda$, or $Q_\alpha$ is $\sigma$-centered, or  $Q_\alpha$ is a sub-Boolean-algebra of the random algebra.
\item[2.] For $\Rel=\RII$: $|Q_\alpha|<\lambda$, or $Q_\alpha$ is $\sigma$-centered.
\item[3.] For $\Rel=\Ra$: $|Q_\alpha|<\lambda$.
\end{itemize}
\end{lemma}

\begin{proof}
$(\Rel,\lambda)$-goodness is preserved by FS ccc iterations (in particular compositions),
as proved in~\cite{MR1071305}, cf.~\cite[6.4.11--12]{BJ}. 
Also, ccc forcings of size ${<}\lambda$ are $(\Rel,\lambda)$-good~\cite[6.4.7]{BJ},
which takes care of the case of Cohens and of $|Q_\alpha|<\lambda$.
So it remains to show that (for $i=1,2$) the ``large'' iterands in the list are
$(\Ri,\lambda)$-good. For $\RI$ 
this follows from~\cite{MR1071305} and~\cite{MR1022984}, cf.~\cite[6.5.17--18]{BJ}.
For $\RII$ this is proven in~\cite{MR1129144}. 
\end{proof}

\begin{lemma}\label{lem:bkb}
Let $\lambda\le\kappa\le\mu$ be uncountable regular cardinals.
After forcing with $\mu$ many Cohen reals $(c_\alpha)_{\alpha\in \mu}$, followed by 
an $(\Rel,\lambda)$-good forcing, we get: For every real $r$ in the final extension, the set 
$\{\beta\in \kappa:\, c_\beta\text{ is unbounded by }r\}$ is cobounded in $\kappa$. I.e., 
$(\exists \alpha\in\kappa)\, (\forall \beta\in \kappa\setminus \alpha)\, \lnot c_\beta\Rel r$.

\end{lemma}

(The 
Cohen real $c_\beta$ can
be interpreted both as Cohen generic element of  $2^\omega$ 
and as Cohen generic element of  $\omega^\omega$; we 
use the interpretation suitable for the relation $\Rel$.)

\begin{proof}
Work in the intermediate extension after $\kappa$ many Cohen reals, let us call it $V_\kappa$.
The remaining forcing (i.e., $\mu\setminus\kappa$ many Cohens composed with the good forcing) is good; so applying Definition~\ref{def:good} we get 
(in $V_\kappa$)
a set $\mathcal{Y}$ of size ${<}\lambda$. 

As the initial Cohen extension is ccc, and $\kappa\ge \lambda$ is regular,
we get some $\alpha\in\kappa$ such that each element $y$ of $\mathcal{Y}$
already exists in the extension by the first $\alpha$ many Cohens, call it
$V_{\alpha}$.
The set of reals $M_y$ bounded by $y$ is meager (and absolute).
Any $c_\beta$ for $\beta\in\kappa\setminus \alpha$
is Cohen over $V_{\alpha}$, and therefore not in $M_y$, i.e., not bounded by $y$.
As this holds for all $y$, $c_\beta$ is unbounded by 
$\mathcal Y$, and thus,
according to the definition of good, 
unbounded by $r$ as well.
\end{proof}

In the light of this result, let us revisit Lemma~\ref{lem:connection} 
with some new notation:
\begin{definition}\label{def:reward}
For $i=1,2,3$, $\lambda>\aleph_0$ regular, and $P$ a ccc forcing notion, let $\myXi(P,\lambda)$  stand for: ``There is a sequence 
$(x_\alpha)_{\alpha\in\lambda}$ of $P$-names such that 
for every $P$-name $y$ we have
$(\exists \alpha\in\lambda)\, (\forall \beta\in \lambda\setminus \alpha)\,P\Vdash \lnot x_\beta \Ri y$.''
\end{definition}

\begin{lemma}\label{lem:goodreward2}
$\myXi(P,\lambda)$ implies $\mathfrak{b}_i\le\lambda$ and $\mathfrak{d}_i\ge\lambda$. In particular:
\begin{itemize}
\item[1.]
$\myXI(P,\lambda)$ implies  $P\Vdash(\,\addnull\le\lambda\,\&\,\cofnull\ge\lambda\,)$.
\item[2.] 
$\myXII(P,\lambda)$ implies  $P\Vdash(\,\covnull\le\lambda\,\&\,\nonnull\ge\lambda\,)$.
\item[3.] 
$\myXIII(P,\lambda)$ implies  $P\Vdash(\,\mathfrak{b}\le\lambda\,\&\,\mathfrak{d}\ge\lambda\,)$.
\end{itemize}
\end{lemma}

\begin{proof}
The set $\{x_\alpha:\, \alpha\in\lambda\}$ is certainly forced to be $\Ri$-unbounded;
and given a set $Y=\{y_j:\, j<\theta\}$ of $\theta<\lambda$ many $P$-names,
each has a bound $\alpha_j$, so for any $\beta\in\lambda$ above all $\alpha_j$ we get
$P\Vdash \lnot x_\beta \Ri y_j$ for all $j$; i.e., $Y$ cannot be dominating.
\end{proof}

\subsection{Ground model Borel functions, partial random forcing}\label{sec:borel}

The following lemma\linebreak seems to be well known (but we are not aware of a good reference or an established notation):

\begin{definition}
Let $Q$ be a forcing notion, and let $\eta$ be a $Q$-name for a real. 
We say that $Q$ is ``generically Borel determined (by $\eta$, via $B$)'', if
\begin{itemize}
\item $Q$ consists of reals,
\item the $Q$-generic filter is determined by the real $\eta$, and moreover:
\item $B\subseteq \mathbb R^2$ is a Borel relation 
such that for all $q\in Q$,
$Q\forces (\,B(q,\eta)\leftrightarrow  q\in G\,)$.
\end{itemize}
\end{definition}

We investigate iterations of such forcings:

\begin{lemma}
Assume that $(P_\beta,Q_\beta)_{\beta<\alpha}$ is a FS ccc iteration such that each $Q_\beta$ is generically Borel determined (in an absolute way already fixed in $V$).
Then for each $P_\alpha$-name $r$ of a real, there is
(in the ground model) a Borel function $F:\mathbb R^\omega\to\mathbb R$ and a sequence $(\alpha_i)_{i\in\omega}$ of ordinals in $\alpha$
such that $P_\alpha$ forces $r=F((\eta_{\alpha_i})_{i\in\omega})$.
\end{lemma}

\begin{proof}
We prove by induction on $\gamma\le\alpha$: 
\begin{itemize}
\item For all $p\in P_\gamma$ 
there is a Borel relation $B^p\subseteq \mathbb{R}^{\omega}$ and a sequence $(\alpha^p_i)_{i\in\omega}$ of elements of $\gamma$ such that $P_\gamma\forces B^p((\eta_{\alpha_i^p})_{i\in\omega})\leftrightarrow p\in G_\gamma$.
\item For each $P_\gamma$-name $r$ of a real, there is a 
Borel function $F^r$ and a sequence $(\alpha^r_i)_{i\in\omega}$ of elements of $\gamma$ such that $P_\gamma\forces r= F^r((\eta_{\alpha_i^p})_{i\in\omega})$.
\end{itemize}
The second item follows from the first, as we can use the 
countable maximal antichains that decide $r(n)=m$.

If $\gamma$ is a limit ordinal, then $P_\gamma$ has no new elements, so there is nothing to do.

So assume $\gamma=\zeta+1$.  
By our assumption,
$Q_\zeta$ is generically Borel determined from $\eta_\zeta$ via
a Borel relation $B_\zeta$. 
Consider $(p,q)\in P_\zeta*Q_\zeta$. This is in $G_\gamma$ iff $p\in G_\zeta$ (which, by induction, is Borel) and $q\in G(\zeta)$.
As $q$ is a real, it is forced that $q=B^q((\alpha^q_i)_{i\in\omega})$. 
Moreover, $P_\zeta$ forces that 
$Q_\zeta$ forces that 
$q\in G(\zeta)$ iff
$B_\zeta(\eta_\zeta,q)$ iff 
$B_\zeta(\eta_\zeta,B^q((\alpha^q_i)_{i\in\omega}))$.
\end{proof}

\begin{definition}
Given $(P_\beta,Q_\beta)_{\beta<\alpha}$ as above, and 
some $w\subseteq \alpha$, we define the $P_\alpha$-name
$\mathbb{R}^w$ to consist of all
reals $r$ such that in the ground model there
are a Borel function $F$ and a sequence $(\alpha_i)_{i\in\omega}$ of elements of $w$ such that $r=F((\eta_{\alpha_i})_{i\in\omega})$.
\end{definition}

The following is straightforward:
\begin{facts}\label{facts:borel}
\begin{itemize}
\item Set (in $V$) $\mu=(|w|+2)^{\aleph_0}$. Then it is forced that
$\mathbb{R}^{w}$ has cardinality ${\le}\mu$.
\item If $w'\supseteq w$, then (it is forced that) $\mathbb{R}^{w'}\supseteq \mathbb{R}^{w}$.
\item If $w$ is the increasing union of $(w_\alpha)_{\alpha\in\gamma}$ with $\cf(\gamma)\ge \omega_1$,
then (it is forced that) $\mathbb{R}^w=\bigcup_{\alpha\in\gamma} \mathbb{R}^{w_\alpha}$.
\item For every $P_\alpha$-name $r$ of a real there is a countable $w$ such that
(it is forced that) $r\in \mathbb{R}^w$.
\end{itemize}
\end{facts}

\begin{definition}
Let $\mathbb{B}$ be (the definition of) random forcing, i.e., 
positive pruned trees $T$, ordered by inclusion.
Given $(P_\beta,Q_\beta)_{\beta<\alpha}$ as above, $w\subseteq \alpha$,
we define the $P_\alpha$-name $\mathbb{B}^w:=\mathbb{B}\cap \mathbb{R}^w$
and call it ``partial random forcing defined from $w$''.
\end{definition}

Clearly $\mathbb{B}^w$ is a subforcing (not necessarily a complete one) of 
$\mathbb{B}$,
and if $p,q$ in $\mathbb{B}^w$ are incompatible in $\mathbb{B}^w$
then they are incompatible in random forcing. In particular $\mathbb{B}^w$
is ccc.

Note that $\mathbb{B}^w$ is forced to be generically Borel determined,
in way already fixed in $V$: The 
generic real $\eta$ is defined by $\{\eta\}=\bigcap\{[s]\in G:\, s\in 2^{{<}\omega}\}$, 
and the Borel relation by ``$\eta\in  [T]$''.

\begin{remark}
In this section, we have provided a very explicit 
notion of ``partial random'', using Borel functions.
The use of Borel functions is not essential, we could use any other 
method of calculating reals from generic reals at certain restricted 
positions, provided this method satisfies Facts~\ref{facts:borel}.
%
One such alternative definition 
has been used in~\cite{MR3513558}: 
We can define the sub-forcing $P_\alpha\restriction w$ of $P_\alpha$ in a natural way, and 
require that it 
is a complete subforcing (which is a closure property of $w$).
Then we can take $Q_\alpha$
to be random forcing as evaluated in the $P_\alpha\restriction w$-extension.

While this approach is basically equivalent (and may seem slightly more natural 
than the artificial use of Borel functions), it has the disadvantage that
we have to take care of the closure property of $w$.
\end{remark}

\begin{definition}
Analogously to ``partial random'', we define the ``partial Hechler''
and ``partial amoeba'' forcings.
\end{definition}
These forcings are generically Borel determined as well.
\subsection{The inital forcing \texorpdfstring{$\Pa$}{P4}}\label{sec:bla2}


Assume that $\lambda $  is regular uncountable and $\mu<\lambda$
implies $\mu^{\aleph_0}<\lambda$.
Then $|w|<\lambda$ implies that the size of a partial forcing
defined by $w$ is ${<}\lambda$.

\begin{definition}
Assume GCH and let $\lambda_1<\lambda_2<\lambda_3<\lambda_4$ be regular cardinals.
Set $\delta_4=\lambda_4+\lambda_4$.
Partition $\delta_4\setminus\lambda_4$ into unbounded sets $S^1$, $S^2$, and $S^3$.
Fix for each $\alpha\in \delta_4\setminus\lambda_4$ some $w_\alpha\subseteq \alpha$ such that each
$\{w_\alpha:\, \alpha\in S^i\}$ is cofinal in $[\delta_4]^{{<}\lambda_i}$.\footnote{I.e.,
if $\alpha\in S^i$ then $|w_\alpha|<\lambda_i$,
and for all $u\subseteq \delta_4$, $|u|<\lambda_i$
there is some $\alpha\in S^i$ with $w_\alpha\supseteq u$.} 

We now define 
$\Pa=(P_\alpha,Q_\alpha)_{\alpha\in\delta_4}$ to be the FS ccc iteration which
first adds $\lambda_4$ many Cohen reals, and such that 
for each $\alpha\in\delta_4\setminus\lambda_4$, 
\[
\text{if $\alpha$ is in}
\left\{
\begin{array}{l}
S^1\\
S^2\\
S^3\\
\end{array}
\right\}\text{, then }
Q_\alpha\text{ is the partial }
\left\{
\begin{array}{c}
\text{amoeba}\\
\text{random}\\
\text{Hechler }\\
\end{array}\right\}
\text{ forcing defined from $w_\alpha$.}
\]

%
\end{definition}

The forcing results in $2^{\aleph_0}=\lambda_4$, which follows from the
following easy and well-known fact:
\begin{lemma}\label{lem:smallbla} 
Let $(P_\alpha,Q_\alpha)_{\alpha<\delta}$
be a FS ccc iteration of length $\delta$ 
such that each 
$Q_\alpha$ is forced to consist of real numbers, 
and set $\lambda(\delta)\coloneq (2+\delta)^{\aleph_0}$.
Then 
$P_\delta\Vdash 2^{\aleph_0}\le \lambda(\delta)$.
\end{lemma}

\begin{proof}
By induction on $\delta$,
we show that there is a dense subforcing of $D_\delta\subseteq P_\delta$
of size ${\le}\lambda(\delta)$. Then the continuum 
has size at most $\lambda(\delta)$ (as
each name of a real corresponds to a countable 
sequence of antichains, labeled with $0,1$, in $P_\delta$,
without loss of generality in $D_\delta$).

For $\delta+1$,  $D_\delta\subseteq P_\delta$ 
is dense and has size ${\le}\lambda(\delta)$,
and $Q_\delta$ is forced to have size ${\le}\lambda(\delta)$.
Without loss of generality we can identify $Q_\delta$ with  a subset
of $\lambda(\delta)$.
Let $D_{\delta+1}$ consist of $(p,\check \alpha)\in P_{\delta+1}$ such that
$p\in D_\delta$ forces $\alpha\in Q_\delta$.

For $\delta$ limit, the union of $D_\alpha$
is dense in $P_\delta=\bigcup_{\alpha\in\delta} P_\alpha$.
\end{proof}

According to Lemma~\ref{lem:gettinggood} $\Pa$ is $(\Ri,\lambda_i)$-good
for $i=1,2,3$, 
so Lemmas~\ref{lem:bkb} and~\ref{lem:goodreward2} gives us:

%
%
%

%


\begin{lemma}\label{lem:goodreward3}
$\myXi(\Pa,\kappa)$ holds for $i=1,2,3$ and each regular cardinal $\kappa$ in $[\lambda_i,\lambda_4]$.

So in particular, $\Pa$ forces $\addnull\le\lambda_1$, $\covnull\le\lambda_2$, $\mathfrak{b}\le\lambda_3$ and $\cofnull=\nonnull=\mathfrak{d}=2^{\aleph_0}$.
\end{lemma}

\begin{theorem}\label{thm:Pa}\cite[Thm.~2]{MR3047455}
$\Pa$ forces $\addnull=\lambda_1$, $\covnull=\lambda_2$, 
$\mathfrak{b}=\lambda_3$, and
$\mathfrak{d}=\lambda_4=2^{\aleph_0}$.
\end{theorem}

\begin{proof}
%
It is easy to see that the partial amoebas take care of $\addnull\ge\lambda_1$:
Let $(N_i)_{i\in\mu}$, $\aleph_1\le\mu<\lambda_1$ be a family of $\Pa$-names of null sets.
Each $N_i$ is a Borel code, i.e., a real, 
and therefore Borel-computed from some countable set $w^i\subseteq \delta_4$.
The union of the $w^i$ is a
set $w^*$ of size ${\le}\mu$ that already Borel-decides all $N_i$.
There is some $\beta\in S^1$ such that $w_\beta\supseteq w^*$, 
so the partial amoeba forcing at $\beta$ sees all the null sets $N_i$ and therefore 
covers their union.

Analogously one proves 
$\covnull\ge\lambda_2$ and
$\mathfrak{b}\ge\lambda_3$.
\end{proof}

We will reformulate the proof for $\covnull$ in a cumbersome manner
that can be conveniently used later on:

\begin{lemma}\label{lem:Pcovnullislarge}
Let $\mySII(P,\lambda,\mu)$ stand for:
``$P$ is a ccc forcing notion, and 
there is a ${<}\lambda$-directed partial order $(S,\prec)$
of size $\mu$
and a sequence $(r_s)_{s\in S}$ of $P$-names for reals
such that for each $P$-name $N$ of a null set 
$(\exists s\in S)\,(\forall t\succ s)\, P\Vdash r_t\notin N$.''
\begin{itemize}
\item $\mySII(P,\lambda,\mu)$ implies 
$P\Vdash(\, \covnull\ge\lambda \,\&\, \nonnull\le \mu\,)$. 
\item $\mySII(\Pa,\lambda_2,\lambda_4)$ holds.
\end{itemize}
\end{lemma}

\begin{proof}
$\covnull\ge\lambda$: 
Fix ${<}\lambda$ many $P$-names $N_\alpha$ of null sets.
Each real has a ``lower bound'' $s_\alpha\in S$, i.e.,
$P\Vdash r_t\notin N_\alpha$ whenever $t\succ s_\alpha$.
Let $t\succ s_\alpha$ for all $\alpha$  (this is possible as
$S$ is directed).  So $P\Vdash r_t\notin N_\alpha$ for every $\alpha$, i.e., the union doesn't cover the reals.

$\nonnull\le\mu$, as the set of all $r_s$ is not null: For every name $N$ of a null set there is some $s\in S$ such that 
$P\Vdash r_s\notin N$.

For $\Pa$, we set $S=S^2$, $s\prec t$ if $w_s\subseteq w_t$, and we 
let $r_s$ be the partial random real added at $s$.
A $\Pa$ name for a null set $N$ depends (in a Borel way) on 
a countable index set $w^*\subseteq \delta_4$.
Fix some $s\in S^2$ such that $w_s\supseteq w^*$,
and pick any $t\succ s$. Then $w_t$ contains all information to calculate
the null set $N$, and therefore the partial random $r_t$ over $w_t$
will avoid $N$.
\end{proof}

\section{The Boolean ultrapower of the forcing}\label{sec:partB}

\subsection{Boolean ultrapowers}

Boolean ultrapowers generalize regular ultrapowers by using arbitrary Boolean algebras
instead of the power set algebra.

\begin{assumption}
$\kappa$ is strongly compact,
$B$ is a $\kappa$-distributive,  $\kappa^+$-cc,
atomless complete Boolean algebra.
\end{assumption}

\begin{lemma}\cite{MR0166107}
Every $\kappa$-complete filter on $B$ can be extended to a $\kappa$-complete ultrafilter $U$.\footnote{For this, neither $\kappa^+$-cc nor atomless is  required,
and it is sufficient that $B$ is $\kappa$-complete.}
\end{lemma}

\begin{proof}
List the required 
properties of $U$ as a set of propositional sentences in $\mathcal L_\kappa$ (a propositional language allowing conjunctions and disjunctions of any size ${<}\kappa$), using 
atomic formulas coding $b\in U$ and $b\notin U$ for $b\in B$.
\end{proof}

\begin{assumption}
$U$ is a $\kappa$-complete ultrafilter on $B$.
\end{assumption}

\begin{lemma}
There is a 
maximal antichain $A_0$ in $B$ of size $\kappa$
such that $A_0\cap U=\emptyset$. In other words, $U$ is not $\kappa^+$-complete.
\end{lemma}
\begin{proof}
Let $A_0$ be a maximal antichain in the open dense set 
$B\setminus U$. As $B$ is $\kappa^+$-cc, $A_0$ has size ${\le}\kappa$. It cannot have size ${<}\kappa$, as $U$ is $\kappa$-complete
and therefore meets every antichain of size ${<}\kappa$.
\end{proof}

The Boolean algebra $B$ can be used as forcing notion. 
As usual, $V$ denotes the universe we start with, sometimes called the
ground model. In the following, we will not actually force with $B$ (or any other p.o.); we always remain in $V$,
but we still use forcing notation.
In particular, we call the usual $B$-names ``forcing names''. 

\begin{definition}
A \textdef{BUP-name} (or: labeled antichain) $x$
is a function $A\to V$ whose domain is a
maximal antichain.
We may write $A(x)$ to denote $A$.
\end{definition}

Each BUP-name corresponds to a forcing-name\footnote{More specifically, to the forcing-name $\{(\widecheck{x(a)},a):\, a\in A(x)\}$.} for an element of $V$. We will identify the BUP-name 
and the corresponding forcing-name. 
In turn, every forcing name $\tau$
for an element of $V$ has a forcing-equivalent BUP-name.

In particular, we can calculate,
for two BUP-names $x$ and $y$, 
the  Boolean value $ \lBrack x=y\rBrack$.\footnote{We can calculate
$\lBrack x=y\rBrack$ more explicitly as follows:
Pick some common refinement $A'$ of $A(x)$ and $A(y)$. 
This defines in an obvious way
BUP-names $x'$ and $y'$ both with domain $A'$:
For $a\in A'$ we set 
$x'(a)=x(\tilde a)$ for $\tilde a$ the unique element of $A(x)$
above $a$. 
Then $\lBrack x=y\rBrack$ is $\bigvee\{a\in A':\, x'(a)=y'(a)\}$ (which is independent of the refinement $A'$).}
\begin{definition}
\begin{itemize}
\item Two BUP-names $x$ and $y$ are \textdef{equivalent},
if $ \lBrack x=y\rBrack\in U$.
\item For $v\in V$, let $\check v$ be a BUP-name-version of 
the standard name for $v$ (unique up to equivalence).
\item 
The \textdef{Boolean ultrapower} $M^-$ 
consists of the equivalence classes $[x]$ of BUP-names $x$;
and we define $[x]\in^- [y]$ by $\lBrack x\in y\rBrack\in U$.
\item  $j^-:V\to M^-$ maps $v$ to $[\check v]$.
\end{itemize}
\end{definition}

We are interested in the $\in$-structure
 $(M^-,\in^-)$.

Given  BUP-names $x_1,\dots,x_n$ and an $\in$-formula $\varphi$,
the truth value
$\lBrack \varphi^V(x_1,\dots,x_n)\rBrack$ is well defined
(it is the weakest element of $B$ forcing that in the ground model $\varphi(x_1,\dots,x_n)$ holds,
which makes sense as $x_1,\dots,x_n$ are guaranteed to be in the ground model).\footnote{Equivalently, we can explicitly calculate 
$\lBrack \varphi^V(x_1,\dots,x_n)\rBrack$ as follows:
Chose a common refinement $A'$ of $A(x_1),\dots, A(x_n)$,
and set  $\lBrack \varphi^V(x_1,\dots,x_n)\rBrack$ to be 
$\bigvee\{a\in A':\, \varphi(x_1'(a),\dots,x_n'(a))\}$;
where again the BUP-names $x_i'$ are the canonically
defined BUP-names with domain $A'$ that are equivalent to $x_i$.}

\begin{lemma}
\begin{itemize}
\item {\L}o{\'{s}}'s theorem: $(M^-,\in^-)\vDash \varphi([x_1],\dots,[x_n])$ iff
$\lBrack \varphi^V(x_1,\dots,x_n)\rBrack \in U$.
\item
$j^-:(V,\in)\to (M^-,\in^-)$ is an elementary embedding.
\item In particular, $(M^-,\in^-)$ is a  ZFC model.

\end{itemize}\end{lemma}
\begin{proof}
Straightforward by the definition of equivalence and of $[x]\in^-[y]$, and by induction (using that 
$U$ is a filter for $\varphi\wedge \psi$ and for $\exists v\, \varphi (v)$, 
and that it is an ultrafilter for $\lnot \varphi$).
For elementarity, note that $M^-\vDash \varphi([\check x_1],\dots,[\check x_n])$ iff
$\lBrack \varphi^V(\check x_1,\dots,\check x_n)\rBrack \in U$ iff $V\vDash 
\varphi(x_1,\dots,x_n)$.
\end{proof}

\begin{lemma}
$(M^-,\in^-)$ is wellfounded.
\end{lemma}
\begin{proof}
This is the standard argument, using the fact that $U$ is $\sigma$-complete:

Assume $[x_{n+1}]\in^-[x_n]$ for $n\in\omega$.
Choose a common refinement $A$ of the antichains $A(x_n)$,
Again, let $x'_n$ be the BUP-names with domain $A$ equivalent to 
$x_n$. 
So, by our assumption, $u_n\coloneq \lBrack x_{n+1}\in x_n\rBrack = \bigvee\{a\in A:\, x'_{n+1}(a)\in x'_n(a)\}$
is in $U$ for each $n$. 
As $U$ is $\sigma$-complete, there is some $u\in U$ stronger than all $u_n$.
This implies: If $a\in A$ is compatible with $u$, then $a$ is compatible with
$u_n$ (for all $n$), and therefore $x'_{n+1}(a)\in x'_n(a)$ for all $n$,
a contradiction.
\end{proof}

\begin{definition}
Let $M$ be the transitive collapse of $(M^-,\in^-)$, and 
let $j:V\to M$ be the composition of $j^-$ with the collapse.
We denote the collapse of $[x]$ by $\BUP{x}$.
So in particular $\BUP{\check v}=j(v)$.
\end{definition}

\begin{lemma}
\begin{itemize}
\item $M\models \varphi(\BUP{x_1},\dots,\BUP{x_n})$ iff $\lBrack \varphi^V(x_1,\dots,x_n)\rBrack \in U$. In particular, $j:V\to M$ is an elementary embedding.
\item 
If $|Y|<\kappa$, then $j(Y)=j''Y$.
In particular, $j$ restricted to $\kappa$ is the identity.
$M$ is closed under ${<}\kappa$-sequences.
\item
$j(\kappa)\neq \kappa$. I.e., $\kappa=\crit(j)$.
\end{itemize}

\end{lemma}
\begin{proof}
If $[x]\in j^-(Y)$, then we can refine the antichain $A(x)$ to some $A'$
such that each $a\in A'$ either forces $x=y$ for some $y\in Y$, or $x\notin Y$.
Without loss of generality (by taking suprema),
we can assume different elements $a$ of $A'$ giving different values $y(a)$; i.e.,
$A'$ has size  $|Y|+1<\kappa$. So $U$ selects an element $a$ of $A'$, and as 
$\lBrack x\in Y \rBrack \in U$,  this element $a$ proves that 
$[x]=j^-(y(a))$.

We have already mentioned that there is a maximal antichain $A_0=\{a_i:\, i\in\kappa\}$  
of size $\kappa$ such that $A_0\cap U=\emptyset$. 
The BUP-name $x$ with $A(x)=A_0$ and $x(a_i)=i$ satisfies $[x]\in^- j^-(\kappa)$,
but is not equivalent to any $\check v$; so $\kappa\le \BUP{x}<j(\kappa)$.
\end{proof}

As we have already mentioned, an arbitrary forcing-name for an element of $V$
has a forcing-equivalent BUP-name, i.e., a maximal antichain labeled
with elements of $V$.
If $\tau$ is a forcing-name for an element of $Y$ ($Y\in V$),
then without loss of generality $\tau$ corresponds to a maximal antichain
labeled with elements of $Y$. We call such an object $y$ a ``BUP-name for an element
of $j(Y)$'' (and not ``for an element of $Y$'', for the obvious reason: unlike
in the case of a forcing extension, $\BUP{y}$ is generally not in $Y$,
but, by definition of $\in^-$, it is in $j(Y)$).

\subsection{The algebra and the filter}

We will now define the concrete Boolean algebra we are going to use:
\begin{definition}
Assume GCH, let $\kappa$ be strongly compact, 
and $\theta>\kappa$ regular.

$P_{\kappa,\theta}$ is the forcing notion adding $\theta$ Cohen subsets of $\kappa$.
More concretely: $P_{\kappa,\theta}$
consists of partial functions from $\theta$ to $\kappa$ with 
domain of size ${<}\kappa$, ordered by extension. 
Let $f^*:\theta\to \kappa$ be the name of the generic function.

$\mathcal B_{\kappa,\theta}$ is the complete Boolean algebra generated by $P_{\kappa,\theta}$.
\end{definition}

Clearly $\mathcal B_{\kappa,\theta}$ is $\kappa^+$-cc and $\kappa$-distributive, as 
$P_{\kappa,\theta}$ is even $\kappa$-closed.


\begin{lemma}\label{lem:embed} 
There is a 
$\kappa$-complete ultrafilter  $U$ on $B=\mathcal B_{\kappa,\theta}$ such that:
\begin{enumerate}
\renewcommand{\theenumi}{\alph{enumi}}
\item
The Boolean ultrapower gives an elementary embedding $j:V\to M$.
$M$ is closed under ${<}\kappa$-sequences.
\item
The elements $\BUP{x}$ of $M$ are exactly (the collapses of equivalence classes of) $B$-names $x$ for elements of $V$; more concretely, a function from 
an antichain (of size $\kappa$) to $V$.
We sometimes say ``$\BUP{x}$ is a mixture of $\kappa$ many possibilities''.

Similarly, for $Y\in V$, the
elements $\BUP{x}$ of $j(Y)$ correspond to the
$B$-names $x$ of elements of $Y$, i.e., antichains labeled with elements of $Y$.
\item\label{item:bla3ww}
If $|A|<\kappa$, then $j''A=j(A)$. In particular, $j$ restricted to $\kappa$ is the identity.
\item\label{item:jokappa} $j$ has critical point $\kappa$, $\cf(j(\kappa))=\theta$,
and $\theta\le j(\kappa)\le \theta^+$.
\item\label{item:joflambda} If $\lambda> \kappa$ is regular, then $\max(\theta,\lambda)\le j(\lambda) <\max(\theta,\lambda)^+$.
\item\label{item:cofinal} If $S$ is a ${<}\lambda$-directed partial order, and $\kappa<\lambda$,
then $j''S$ is cofinal in $j(S)$.
\item\label{item:cofinalities}
If $\cf(\alpha)\neq \kappa$, then
$j''\alpha$ is cofinal in $j(\alpha)$, so in particular
$\cf(j(\alpha))=\cf(\alpha)$.
\end{enumerate}
\end{lemma}

\begin{proof}
We have already seen (a)--(c).

(\ref{item:jokappa}):
For each $\delta\in\theta$, 
$f^*(\delta)$ is a forcing-name for an element of $\kappa$, and thus a BUP-name
for an element of $j(\kappa)$.
Let $x$ be some other BUP-name for an element of $j(\kappa)$,
i.e., an antichain $A$ of size $\kappa$
labeled  with elements of $\kappa$. 
Let $\delta\in\theta$ be bigger than
the supremum of $\supp(a)$ for each $a\in A$.
We call such a pair $(x,\delta)$ ``suitable'', and set
$b_{x,\delta}\coloneq \lBrack  f^*(\delta)>x\rBrack $.
We claim that all these elements form a basis for a  $\kappa$-complete filter.
To see this, fix suitable pairs
$(x_i,\delta_i)$ for $i<\mu$ where $\mu<\kappa$; we have to show that $\bigwedge_{i\in\mu} b_{x_i,\delta_i}\neq \mathbb{0}$. Enumerate $\{\delta_i:\, i\in \mu\}$ increasing (and without repetitions)
as $\delta^j$ for $j\in\gamma\le \mu$. 
Set $A_j=\{i:\, \delta_i=\delta^j\}$.
Given $q_j$,  
define $q_{j+1}\in P_{\kappa,\theta}$ as follows:
$q_{j+1}\le q_j$; $\delta^j\in \supp(q_{j+1})\subseteq \delta^j\cup\{\delta^j\}$; and
$q_{j+1}\restriction \delta^j$ 
decides for all $i\in A_j$ 
the values of $x_i$ to be some $\alpha_i$; and 
$q_{j+1}(\delta^j)=\sup_{i\in A_j}(\alpha_i)+1$.
For $j\le\gamma$ limit, let $q_j$ be the union of $\{ q_k:\, k<j\}$. Then $q_{\gamma}$ is stronger than each $b_{x_i,\delta_i}$.

As $\kappa$ is strongly compact, we can extend the $\kappa$-complete filter generated by all $b_{x_i,\delta_i}$
to a $\kappa$-complete ultrafilter $U$.
Then the sequence $(\BUP{f^*(\delta)})_{\delta\in \theta}$ 
is strictly increasing (as $(f^*(\delta),\delta')$ is suitable
for all $\delta<\delta'$) and cofinal in $j(\kappa)$ (as we have just seen); 
so $\cf(j(\kappa))=\theta$.

(\ref{item:joflambda}):
We count all BUP-names for elements of $j(\lambda)$. As we can assume
that the antichains are subsets of $P_{\kappa,\theta}$, which has size
$\theta$, and as $\lambda$ is regular and GCH holds, 
we get $|j(\lambda)|\le [\theta]^{\kappa}\times \lambda^{\kappa}=\max(\theta,\lambda)$.

(\ref{item:cofinal}): An element $\BUP{x}$ of $j(S)$ is a mixture 
of $\kappa$ many possibilities in $S$. 
As $\kappa<\lambda$, there is some $t\in S$
above all the possibilities. Then $j(t)>\BUP{x}$.

(\ref{item:cofinalities}): Set $\mu=\cf(\alpha)$,  and pick an increasing cofinal sequence $\bar\beta=(\beta_i)_{i\in\mu}$ in $\alpha$.
$j(\bar\beta)$ is increasing cofinal in $j(\alpha)$ (as this is absolute between $M$ and $V$). If $\mu<\kappa$, then $j''\bar\beta=j(\bar\beta)$, otherwise use (f).
\end{proof}

\subsection{The ultrapower of a forcing notion}\label{ss:a3}

We now investigate the relation of a forcing notion $P\in V$
and its image $j(P)\in M$, which we use 
as a forcing notion over $V$. (Think of $P$ 
as being one of the forcings of Section~\ref{sec:partA}; it has no relation with the Boolean algebra $B$.)

Note that as $j(P)\in M$ and $M$ is transitive, every $j(P)$-generic filter $G$ over
$V$ is trivially generic over $M$ as well, and we will use absoluteness between $M[G]$ and $V[G]$ to prove various properties of $j(P)$.

\begin{lemma} If $P$ is $\kappa$-cc, then
$j$ gives a complete embedding from $P$ into $j(P)$. I.e.,
$j'' P$ is a complete subforcing of $j(P)$, and $j$ is an isomorphism
from $P$ to $j'' P$.
\end{lemma}

\begin{proof}
It is clear that $j$ is an isomorphism onto $j''P$:
By definition the order $<_{j(P)}$ on $j(P)$ is $j(<_P)$, 
and by elementarity $p\le_P q$
iff $j(q)<_{j(P)} j(p)$. Also, $p\perp q$ is preserved:
$M\vDash p\perp_{j(P)}q$ by elementarity, so
$p\perp_{j(P)}q$ holds in $V$ (as $j(P)\in M$ and $M$ is transitive).

It remains to be shown that each maximal antichain $A$ of $P$ 
is preserved, i.e., $j''A\subseteq j(P)$ is predense.

By our assumption, $|A|<\kappa$, so $j''A = j(A)$ (by Lemma~\ref{lem:embed}(\ref{item:bla3ww})), which is maximal in $M$ (by
elementarity) and thus maximal in $V$ (by absoluteness).
\end{proof}

Accordingly, we can canonically translate $P$-names into $j(P)$-names, etc.

For later reference, let us make this a bit more explicit:
Let $g$ be a $P$-name for a real (i.e., an element of $\omega^\omega$).
Each $g(n)$ is decided by a maximal antichains $A_n$,
where $a\in A_n$ forces $g(n)=g_{n,a}\in\omega$. Then the $j(P)$-name
$j(g)$ corresponds to the antichains 
\begin{equation}\label{eq:bla}
j(A_n)=j'' A_n\text{, and }j(a)\text{ forces }
j(g)(n)=g_{n,a}\text{ for each }a\in A_n.
\end{equation}

\begin{lemma}
If $P=(P_\alpha,Q_\alpha)_{\alpha<\delta}$ is a finite support (FS) ccc iteration of length $\delta$, then 
$j(P)$ is a FS ccc iteration of length $j(\delta)$ (more formally: it is canonically equivalent to one).
\end{lemma}

\begin{proof}
$M$ certainly thinks that $j(P)=(P^*_\alpha,Q^*_\alpha)_{\alpha<j(\delta)}$ is a
FS iteration of length $j(\delta)$.

By induction on $\alpha$ we define the 
FS ccc iteration $(\tilde P_\alpha,\tilde Q_\alpha)_{\alpha<j(\delta)}$ and show that 
$P^*_\alpha$ is a dense subforcing of  $\tilde P_\alpha$:
Assume this is already the case for $P^*_\alpha$.
$M$ thinks that $Q^*_\alpha$ is a $P^*_\alpha$-name,
so we can interpret it as a $\tilde P_\alpha$-name
and use it as $\tilde Q_\alpha$.
Assume that $(p,q)$ is an element (in $V$) of $\tilde P_\alpha*\tilde Q_\alpha$.
So $p$ forces that $q$ is a name in $M$; we can increase $p$
to some $p'$ that decides $q$ to be the name $q'\in M$. 
By induction we can further increase $p'$ to $p''\in P^*_\alpha$,
then $(p'',q')\in P^*_{\alpha+1}$ is stronger than $(p,q)$.
(At limits there is nothing to do, as we use FS iterations.)

$j(P)$ is ccc, as any $A\subseteq j(P)$ of size $\aleph_1$ is 
in $M$ (and $M$ thinks that $j(P)$ is ccc).
\end{proof}

Similarly, we get:
\begin{itemize}
\item 
If $\tau=\BUP{x}$ is in $M$ a $j(P)$-name for an element of $j(Z)$,
then $\tau$ is a mixture of $\kappa$ many $P$-names for an element of $Z$
(i.e., the BUP-name $x$ consists  of
an antichain $A\subseteq B$ labeled, without loss of generality, with $P$-names for elements of $Z$).

(This is just the instance of ``each $\BUP{x}\in j(Y)$
is a mixture of elements of $Y$'', where we set $Y$ to be the
set\footnote{Formally: We set $Y$ to be some set that contains representatives of each equivalence class of $P$-names of elements of $Z$.}  of $P$-names for elements of $Z$.)

\item 
A $j(P)$-name $\tau$
for an element of $M[G]$ has an equivalent $j(P)$-name in $M$.

(There is a maximal antichain $A$ of $j(P)$
labeled with $j(P)$-names in $M$. 
As $M$ is countably closed, this labeled antichain is in $M$, 
and gives a $j(P)$-name in $M$ equivalent to $\tau$.)

\item

In $V[G]$, $M[G]$ is closed under ${<}\kappa$ sequences.

(We can assume the names to be in $M$ and use ${<}\kappa$-closure.)

\item
In particular, every $j(P)$-name for a real, a Borel-code, 
a countable sequence of reals, etc., is in $M$ (more formally: 
has an equivalent name in $M$).

\item If each iterand is forced to consist of reals, 
then $j(P)$ forces the continuum to have size at most $|2+j(\delta)|^{\aleph_0}$.

(This follows from Lemma~\ref{lem:smallbla}
as $j(P)$ also satisfies that 
each iterand consists of reals.)
\end{itemize}

\subsection{Preservation of values of characteristics}

\begin{lemma}\label{lem:simple}
Let $\lambda$  be a regular uncountable cardinal and $P$ a ccc forcing.
\begin{enumerate}
\renewcommand{\theenumi}{\alph{enumi}}
\item Let $\mathfrak{x}$ be either $\addnull$ or $\mathfrak{b}$.
If $P\Vdash \mathfrak{x} = \lambda$ and $\kappa\neq \lambda$, then
$j(P)\Vdash \mathfrak{x} = \lambda$.
\item Let $\mathfrak{y}$ be either $\cofnull$ or $\mathfrak{d}$.
If $P\Vdash\mathfrak{y}\ge\lambda$  and $\kappa<\lambda$, then $j(P)\Vdash\mathfrak{y}\ge\lambda$.
\item Let $(\mathfrak{x},\mathfrak{y})$ be either
$(\mathfrak{b},\mathfrak{d})$ or $(\addnull,\cofnull)$.
Then we get:
\\
If $P\Vdash 
(\,\kappa < \mathfrak{x}\,\&\, \mathfrak{y}\le \lambda\,)$ 
then $j(P)\Vdash\mathfrak{y}\le\lambda$.
\end{enumerate}
\end{lemma}

\begin{proof}
\textbf{(a)} We formulate the proof for $\addnull$; the  proof for $\mathfrak{b}$ is the same.

Let $\bar N=(N_i)_{i<\lambda}$ be $P$-names for an increasing sequence of null sets such that $\bigcup_{i<\lambda} N_i$ is not null.
So in particular
for every $P$-name $N$ of a null set: 
$(\exists i_0\in\lambda)\, (\forall i\in \lambda\setminus i_0)\, P\Vdash N_i \nsubseteq N$.
(We can choose the $i_0$ in $V$ due to ccc.)

Therefore $M$ thinks that the same holds for the sequence $j(\bar N)$
of $j(P)$-names of length $j(\lambda)$.
So whenever $N$ is a $j(P)$-name of a null set,
we can assume without loss of generality that $N\in M$, 
so $M$ thinks that from some $i_0$ on it is forced that $N_i\nsubseteq N$,
which is absolute.


As $\kappa\neq \lambda$, we know that $j''\lambda$ is cofinal in $j(\lambda)$.
So (since the sequence $j(\bar N)$ is increasing) we
can use $(j(N_i))_{i\in \lambda}$ and get the same property.

This shows that $j(P)\Vdash \addnull\le \lambda$

For the other inequality, fix some 
$\chi< \lambda$, and 
$(N_i)_{i<\chi}$ a family of $j(P)$-names for null sets (without loss of generality each name is in $M$), and $p\in j(P)$.
\begin{itemize}
\item Case 1: $\kappa\ge \lambda$.
Then the sequence $(N_i)_{i<\chi}$ (as well as $p$) is in $M$, and $M\models\left( p\Vdash \bigcup N_i \text{ null}\right)$; which is absolute.
\item Case 2: $\kappa<\lambda$. 
Every $N_i$ is a ``mixture''  of $\kappa$ many $P$-names for null sets, so there is a single $P$-name $N'_i$ such that $P$ forces $N'_i$ is superset of all the names involved. Therefore, $j(P)$ forces that $j(N'_i)\supseteq N_i$.
And $P$ forces that $\bigcup_{i<\chi} N'_i$ is null, i.e., covered by some null set
$N^*$. Then $j(P)$ forces that $j(N^*)$ covers $\bigcup_{i<\chi} N_i$.
\end{itemize}

\textbf{(b)} We show that a small set cannot be dominating:
Fix a sequence $(f_i)_{i<\chi}$ of $j(P)$-names of reals, with $\chi<\lambda$.
Each $f_i$ 
corresponds to $\kappa<\lambda$ many possible $P$-names.
As $\chi<\lambda$, there is a 
$P$-name $g$ unbounded by all $\chi\times \kappa<\lambda$ many 
possible $P$-names.
So if $f$ is any of the possibilities, then $P$ 
forces $g\nle^* f$; and thus 
$j(P)$ forces $j(g) \nle^* f_i$ for all $i$.
So $j(P)$ forces $\mathfrak{d}\ge\lambda$. 

The same proof
works for $\cofnull$ (using ``the null set $g$ is not a subset of any 
of the possible null sets'').

\textbf{(c)}
For $(\mathfrak{x},\mathfrak{y})=(\mathfrak{b},\mathfrak{d})$:
Fix  a $P$-name of a dominating family $\bar f=(f_i)_{i\in\lambda}$.

We claim that $j(P)$ forces that $j''\bar f=(j(f_i))_{i<\lambda}$ is dominating.
Let $r$ be a $j(P)$-name of a real, i.e., 
a mixture of $\kappa$ many possibilities (each possibility 
corresponding to a $P$-name for a real).
As $P\Vdash \kappa<\mathfrak{b}$, 
$P$ forces that these reals cannot be unbounded, i.e., there is a $P$-name $\alpha\in\lambda$
such that $f_\alpha$ is forced 
to dominate all the possibilities.
By absoluteness, $j(P)\Vdash j(f_\alpha)>^* r$.

It remains to be shown that $j(P)\Vdash j(f_\alpha)\in j''\bar f$.
(Note that $\alpha$ is just a $P$-name.)
Fix a maximal antichain $A$ in $P$ deciding $\alpha$,
i.e., $a\in A$ forces $\alpha=\alpha(a)$.
As $j$ maps $P$ completely into $j(P)$,
$j''A$ is a maximal antichain in $j(P)$.
So $j(P)$ forces that exactly on $j(a)$ for $a\in  A$
is in the generic filter, cf.\ \eqref{eq:bla}.
Accordingly $j(f_\alpha)=j(f_{\alpha(a)})\in j''\bar f$.

The proof for $\cofnull$ is the same.
\end{proof}

For the other direction of the invariants, and the pair $(\covnull, \nonnull)$, 
we use the following two lemmas, which are reformulations of results of Shelah.\footnote{S.~Shelah, personal communication.}


Recall Definition~\ref{def:reward} (which is useful
because of Lemma~\ref{lem:goodreward2} and satisfied for the inital forcing 
according to Lemma~\ref{lem:goodreward3}).
\begin{lemma}\label{lem:covnullsmallpreserved}
Assume $\myXi(P,\lambda)$. Then $\myXi(j(P),\cf(j(\lambda)))$.
\\
So if $\kappa\ne \lambda$, then $\myXi(j(P),\lambda)$, and if $\kappa=\lambda$, then $\myXi(j(P),\theta)$.
\end{lemma}

\begin{proof}
Let $\bar y=(y_\alpha)_{\alpha<\lambda}$ be the sequence of $P$-names
witnessing $\myXi(P,\lambda)$. Note that $j(\bar y)$ is a sequence of length
$j(\lambda)$; we denote the $\beta$-th element by $(j(\bar y))_\beta$.
So $M$ thinks:
For every $j(P)$-name $r$ of a real 
$(\exists \alpha\in j(\lambda))\, (\forall \beta \in j(\lambda)\setminus \alpha)\, \lnot (j(\bar y))_\beta \Ri r$.
This is absolute. In particular, pick in $V$ a cofinal subset $A$ of $j(\lambda)$ 
of order type $\cf(j(\lambda))=:\mu$. Then $j(\bar y)\restriction A$ witnesses that 
$\myXi(j(P),\mu)$ holds. 
\end{proof}

We have seen in Lemma~\ref{lem:Pcovnullislarge} that 
$\mySII(\Pa,\lambda_2,\lambda_4)$ holds and implies that 
$\Pa$ forces $\covnull\ge \lambda_2$ and $\nonnull\le\lambda_4$ (the latter being  
trivial in the case of $\Pa$).

\begin{lemma}\label{lem:covnulllargepreserved}
Assume $\mySII(P,\lambda,\mu)$. If $\kappa> \lambda$,
then $\mySII(j(P),\lambda,|j(\mu)|)$;
if 
$\kappa< \lambda$,
then $\mySII(j(P),\lambda,\mu)$.
\end{lemma}

\begin{proof}
Let $(S,\prec)$ and $\bar r$ witness $\mySII(P,\lambda,\mu)$.

$M$ thinks that
\begin{multline}\tag{$*$}\label{eq:ff}
\text{for each $j(P)$-name $N$ of a null set }\\
(\exists s\in j(S))\,(\forall t\in j(S))\, t\succ s\rightarrow  j(P)\Vdash (j(\bar r))_t\notin N,
\end{multline}
which is absolute.

If $\kappa>\lambda$, then $j(\lambda)=\lambda$,  and $j(S)$ is $\lambda$-directed in $M$ and therefore in $V$ as well, and so we get $\mySII(j(P),\lambda,|j(\mu)|)$.

So assume $\kappa<\lambda$.
We claim that $j''S$ and $j''\bar r$ witness $\mySII(j(P),\lambda,\mu)$.
$j''S$ is isomorphic to $S$, so directedness is trivial. 
Given a $j(P)$-name $N$, without loss of generality in $M$, there is 
in $M$ a bound $s\in j(S)$ as in~\eqref{eq:ff}. As $j''S$ is cofinal in $j(S)$ (according
to Lemma~\ref{lem:embed}(\ref{item:cofinal})), there is some $s'\in S$ such that $j(s')\succ s$.
Then for all $t'\succ s'$, i.e., $j(t')\succ j(s')$, we get $j(P)\Vdash j(r_t)\notin N$.
\end{proof}

\subsection{The main theorem}

We now have  everything required for the main result:

\begin{theorem}
Assume GCH and that $\aleph_1<\kappa_7<\lambda_1<\kappa_6<\lambda_2<\kappa_5<\lambda_3<\lambda_4<\lambda_5<\lambda_6<\lambda_7$
are regular, $\kappa_i$ strongly compact for $i=5,6,7$.
Then there is a ccc order $\PaVII$ forcing
\begin{multline*} 
\addnull=\lambda_1 < \covnull =\lambda_2 <\mathfrak{b}=\lambda_3 <\\< \mathfrak{d}=\lambda_4 < \nonnull=\lambda_5 < \cofnull=\lambda_6 < 2^{\aleph_0}=\lambda_7.
\end{multline*} 
\end{theorem}


\begin{proof}

Let $j_i: V\to M_i$ be the Boolean ultrapower embedding with
$\cf(j(\kappa_i))=\lambda_i$ (for $i=5,6,7$).
Recall that $\Pa$ is an iteration of length $\delta_4$. We set 
$\PaV\coloneq j_5(\Pa)$, $\PaVI\coloneq j_6(\PaV)$, and $\PaVII\coloneq j_7(\PaVI)$;
and 
$\delta_5\coloneq j_5(\delta_4)$,
$\delta_6\coloneq j_6(\delta_5)$ and $\delta_7\coloneq j_7(\delta_6)$.


It is enough to show the following:

\begin{enumerate}
\renewcommand{\theenumi}{\alph{enumi}}
\item\label{item:ccc} 
$\Pai$ is a FS ccc iteration of length $\delta_i$ and forces  $2^{\aleph_0}=\lambda_i$
for $i=4,5,6,7$.
\item\label{item:simple} $\Pai\Vdash (\, \addnull=\lambda_1 \,\&\,
\mathfrak{b}=\lambda_3\,\&\,
\mathfrak{d}=\lambda_4\,)$ for $i=4,5,6,7$.
\item\label{item:large}
$\Pai\Vdash \nonnull\ge \lambda_5$ for $i=5,6,7$.\\
$\Pai\Vdash \cofnull\ge \lambda_6$ for $i=6,7$.\\
$\Pai\Vdash \covnull\le\lambda_2$ for $i=4,5,6,7$.
\item\label{item:d}
$\Pai\Vdash \cofnull= \lambda_6$ for $i=6,7$.
%
\item\label{item:covnulllarge} $\Pai \vDash 
(\, \covnull\ge \lambda_2 \, \& \, \nonnull\le \lambda_5\,)$ for
$i=4,5,6,7$.
\end{enumerate}

(\ref{item:ccc}) was shown in Section~\ref{ss:a3}. 

(\ref{item:simple}): For $\Pa$ this is~Theorem~\ref{thm:Pa}.
For $\PaV$ use Lemma~\ref{lem:simple} 
(using for $\mathfrak{d}$ that $\kappa_5<\lambda_3$).
Using the same lemma again we get the result for $\PaVI$ and $\PaVII$
(using that $\kappa_i<\lambda_3$ for $i=6,7$ as well.)

(\ref{item:large}):
As $\kappa_5>\lambda_2$, we have
$\myXII(\Pa,\kappa_5)$ (by Lemma~\ref{lem:goodreward3}),
and thus $\myXII(\PaV,\lambda_5)$ (by Lemma~\ref{lem:covnullsmallpreserved},
as $\cf(j_5(\kappa_5))=\lambda_5$),
so $\PaV\Vdash \nonnull\ge\lambda_5$ (Lemma~\ref{lem:goodreward2}).
Repeating the same argument we get $\myXII(\Pai,\lambda_5)$ for $i=6,7$
(as $\kappa_i\neq \lambda_5$ for $i=6,7$).

Analogously, as $\kappa_6>\lambda_1$, we start with 
$\myXI(\Pa,\kappa_6)$, get  
$\myXI(\PaV,\kappa_6)$  (as $\kappa_5\neq\kappa_6$)
and then $\myXI(\PaVI,\lambda_6)$ (as $\cf(j_6(\kappa_6))=\lambda_6$) and
$\myXI(\PaVII,\lambda_6)$ (again as $\kappa_7\neq\lambda_6$). 
So we get  thus
$\Pai\Vdash \cofnull\ge\lambda_6$ for $i=6,7$.

Similarly, 
$\myXII(\Pa,\lambda_2)$ holds, which is preserved by all embeddings, so we get
$\covnull\le \lambda_2$.

(\ref{item:d}):
As $\PaVI$ forces the continuum to have size $\lambda_6$, 
the previous item implies $\PaVI\Vdash \cofnull=\lambda_6$.
And as in (\ref{item:simple}), this implies the same for $\PaVII$ (as $\kappa_7<\lambda_1$, the
value of $\addnull$).

(\ref{item:covnulllarge}): 
$\mySII(\Pa,\lambda_2,\lambda_4)$  holds (cf.~Lemma~\ref{lem:Pcovnullislarge}).
So by Lemma~\ref{lem:covnulllargepreserved} for the case $\kappa>\lambda$,
and as $|j_5(\lambda_4)|=\lambda_5$, according to~Lemma~\ref{lem:embed}(\ref{item:joflambda}),
$\mySII(\PaV,\lambda_2,\lambda_5)$ holds. I.e., $\PaV$ forces $\covnull\ge \lambda_2$ and $\nonnull\le\lambda_5$
(the latter being trivial as the continuum has size $\lambda_5$).
For $i=6,7$, the same lemma, now for the case $\kappa<\lambda$, gives
$\mySII(\Pai,\lambda_2,\lambda_5)$, i.e., $\Pai$ forces $\covnull\ge \lambda_2$ and  $\nonnull
\le \lambda_5$.
\end{proof}

\subsection{An alternative}

In the same way we can prove the consistency of
\[
\aleph_1 < \addnull < \covnull < \nonmeager < \covmeager < \nonnull < \cofnull < 2^{\aleph_0}.\]
(I.e., we can replace $\mathfrak{b}$ and $\mathfrak{d}$ by $\nonmeager$ and $\covmeager$, respectively.)

For this, we use the following relation as $\Ra$:
\[
f \Ra g\text{, if  }f,g\in \omega^\omega\text{ and }(\forall^* n\in\omega)\, f(n)\neq g(n).
\]
By a result of~\cite{MR671224,MR917147} (cf.~\cite[2.4.1~and~2.4.7]{BJ}) we have
\[
\nonmeager=\mathfrak{b}_3\text{ and }
\covmeager=\mathfrak{d}_3.
\]
As before, we use that an iteration where each iterand has size ${<}\lambda_3$
is $(\Ra,\lambda_3)$-good.

To define $\Pa$, we use partial eventually different (instead of 
partial Hechler) forcings.

Unlike for $(\mathfrak{b},\mathfrak{d})$, we do not know whether
$\nonmeager=\lambda$ is generally preserved if $\kappa\neq\lambda$
and $\covmeager=\lambda$ is preserved if $\kappa$ is small;
but we can use the same argument for
$(\nonmeager,\covmeager)$ that we have used for $(\covnull,\nonnull)$.
So we can get the analog of
Lemma~\ref{lem:Pcovnullislarge}
that proves that 
$\nonmeager$ is large and 
$\covmeager$ small; and $\myXIII$ 
implies that $\nonmeager$ is small and 
$\covmeager$ large.

\bibliographystyle{alpha}
\bibliography{8values}
\end{document}